\documentclass[12pt]{amsart}
\usepackage{amsmath, amscd, amssymb, amsthm,amsfonts}
\usepackage[nohug]{diagrams}
\usepackage[T1]{fontenc}
\makeatletter

\usepackage{dbnsymb}
\usepackage{mthsymb7}
\usepackage{stmaryrd}

\ifx\pdfoutput\undefined
\usepackage{graphicx}
\else

\usepackage[pdftex]{graphicx}
\usepackage{epstopdf}
\fi

\DeclareMathOperator{\Der}{D}
\DeclareMathOperator{\perf}{perf}

\DeclareMathOperator{\Ho}{Ho}
\DeclareMathOperator{\tr}{Tr}

\DeclareMathOperator{\Cob}{Cob}
\DeclareMathOperator{\PCob}{Pre-Cob}

\DeclareMathOperator{\Tr}{Tr}

\DeclareMathOperator{\Spec}{Spec}

\DeclareMathOperator{\Kom}{Kom}
\DeclareMathOperator{\Mat}{Mat}

\DeclareMathOperator{\Cone}{Cone}

\DeclareMathOperator{\TL}{TL}
\DeclareMathOperator{\TV}{TV}

\DeclareMathOperator{\K}{K}

\DeclareMathOperator{\SU}{SU}
\DeclareMathOperator{\SO}{SO}

\DeclareMathOperator{\Ktheory}{K}

\newcommand{\inp}[1]{\ensuremath{\langle #1 \rangle}}
\newcommand{\inpp}[1]{\ensuremath{\langle\langle #1 \rangle\rangle}}
\newcommand{\module}[0]{\operatorname{-mod}}

\newcommand{\cal}[1]{\ensuremath{\mathcal{#1}}}

\newcommand{\normaltext}[1]{\textnormal{#1}}

\DeclareMathOperator{\Id}{Id}

\newcommand{\CPicTailbox}[1]{\Bigg[\CPic{#1}\hspace{-.04in}\Bigg]}
\newcommand{\CPPicTailbox}[1]{\Bigg[\hspace{-.0in}\CPPic{#1}\hspace{-.04in}\Bigg]}

\newcommand{\CPic}[1]{
\begin{minipage}{.45in}
\includegraphics[scale=.75]{#1}
\end{minipage}
}

\newcommand{\CPPic}[1]{
\begin{minipage}{.8in}
\includegraphics[scale=1.1]{#1}
\end{minipage}
}

\newcommand{\MPic}[1]{
\begin{minipage}{.35in}
\includegraphics[scale=.45]{#1}
\end{minipage}
}

\newcommand{\BPic}[1]{
\begin{minipage}{1in}
\includegraphics[scale=1.5]{#1}
\end{minipage}
}

\theoremstyle{plain}
\newtheorem{theorem}[subsection]{Theorem}

\newtheorem{proposition}[subsection]{Proposition}

\newtheorem{lemma}[subsection]{Lemma}
\theoremstyle{remark}

\newtheorem*{remark}{Remark}
\theoremstyle{definition}

\newtheorem{definition}[subsection]{Definition}

\numberwithin{equation}{section}

\begin{document}
\title[Handle slides and Localizations of Categories]{Handle slides and Localizations of Categories}

\author[Benjamin Cooper and Vyacheslav Krushkal]{Benjamin Cooper and Vyacheslav Krushkal}

\address{Department of Mathematics, University of Virginia, Charlottesville, VA 22904}
\email{bjc4n\char 64 virginia.edu, krushkal\char 64 virginia.edu}

\begin{abstract}
We propose a means by which some categorifications can be evaluated at a root of
unity. This is implemented using a suitable localization in the context of
prior work by the authors on categorification of the Jones-Wenzl
projectors.  Within this construction we define objects, invariant under
handle slides, which decategorify to the $\SU(2)$ quantum invariants at low
levels.
\end{abstract}

\maketitle

\section{Introduction}

Constructions of invariants of 3-manifolds associated to Chern-Simons theory
rely on the specialization of a parameter $q$ at a root of unity \cite{RT,
  TV}. In Khovanov's categorification of the Jones polynomial \cite{Kh} the
variable $q$ is represented by an extra grading. Understanding how to
``set'' this grading to a root of unity has been a difficult problem in the
program to derive invariants of 3-manifolds from categorifications of knot
polynomials. In this paper we propose a means by which such
categorifications can be evaluated at a root of unity. This proposal is
motivated in section \ref{handleslide} by a construction of objects which
are invariant under handle slides.

In the Temperley-Lieb algebra the evaluation of $q$ at a root of unity is
closely related to taking a suitable quotient by the Jones-Wenzl projector
$p_N$.  In \cite{CK} the authors constructed chain complexes $P_N$ within
the universal categorification of the Jones polynomial (see \cite{KH, DBN})
that become Jones-Wenzl projectors in the image of the Grothendieck
group. These chain complexes are uniquely characterized up to homotopy by
the defining properties of the Jones-Wenzl projectors $p_N$. In this paper
we consider a quotient of the homotopy category of chain complexes by the
projector $P_N$. The Verdier quotient classically considered in homological
algebra turns out to be unsuitable in this context (see section
\ref{verdier}). We propose a different notion in section \ref{cosetcon}
called the coset construction. In this context we define objects which
categorify Reshetikhin-Turaev invariants at low levels.

We will now outline a few more details of the construction.  The
Temperley-Lieb algebra has coefficients in $\mathbb{Z}[q,q^{-1}]$.  The
evaluation of a polynomial $f(q)$ at a root of unity may be implemented
algebraically by passing to the quotient $\pi(f)$ where $\pi : \mathbb{Z}[q]
\to \mathbb{Z}[q]/(\varphi_p(q))$ is the quotient map and $\varphi_p(q)$ is
the $p$th cyclotomic polynomial. For example, when $p$ is a prime number
$$\varphi_p(q) = q^{p-1} + \cdots + q + 1.$$

Given an element $e\in \TL$ such that $\varphi_p(q) | \Tr(e)$, the ideal
generated by $\varphi_p(q)$ contains the ideal generated by
$\Tr(e)$. Therefore the quotient obtained by setting $q$ to be a $p$th root
of unity factors through the quotient of $\TL$ by $e$:

$$ \TL \longrightarrow  \TL/\inp{e}  \longrightarrow \TL/\inp{\varphi_p(q)}.$$

When $N=p-1$ the Jones-Wenzl projectors $p_N \in \TL_N$ provide a choice of
$e$.  In particular, when $p$ is a prime number the trace of the $N$th
projector $\Tr(p_{N})$ is given by the quantum integer $[p]$ and

$$q^{p-1} [p] = q^{2(p-1)} + \cdots + q^2 + 1 = \varphi_p(q)\varphi_p(-q).$$

Since the graded Euler characteristic of the categorified projector agrees
with the Jones-Wenzl projector,  if a quotient $\Ho(\Kom)/(P_{N})$ could be defined it would be closely
related a desired evaluation of $\Ho(\Kom)$ at a root of unity. In this
paper we prove that there is an interesting quotient of this form and
motivate the claim that this is sufficient for topological considerations by
identifying structures which are familiar from low dimensional topology.

Different approaches to categorification of $3$-manifold invariants have been proposed in \cite{MW, Rozansky1}. These authors consider $3$-manifold invariants in different contexts, not involving evaluation at a root of unity.

Sections \ref{TL spaces} and \ref{categorified TL section} review
constructions of $SU(2)$ TQFTs and the relevant material on categorification
from \cite{DBN} and \cite{CK}. We introduce a filtration
on coefficients corresponding to the level. The emphasis in both TQFT and categorified
settings is on the Jones-Wenzl projectors.
In section \ref{cosetcon} we define a suitable version of localization.
The unusual nature of
this construction is motivated in section \ref{verdier} by examining the
structure of Verdier localization. Section \ref{handleslide}
defines objects $\Omega$ which are invariant under handle slides.

\section{Temperley-Lieb Spaces and a construction of $\SU(2)$ TQFTs}\label{TL spaces}

This section summarizes parts of the constructions of $SU(2)$ quantum invariants which are relevant in this paper, the Turaev-Viro theory
associated
to a surface and the Reshetikhin-Turaev invariant of $3$-manifolds based on a surgery presentation. The general outline presented here
follows the ``picture TQFT'' approach of \cite{FNWW, Walker}. Our exposition differs from those available in the literature in the choice
of the coefficient ring, and the main point of this section is contained in \ref{arelcyc} which shows that a version of the TQFT may be defined using the Jones-Wenzl
projectors, without a reference to a root of unity.

A compact oriented surface $\Sigma$ with boundary is \emph{labelled} when
paired with a map $\phi : \partial\Sigma \to \coprod (S^1, \hat{x})$ which
identifies each boundary circle with a model circle containing a fixed
set of marked points $\hat{x}$.

\begin{definition}\label{TL def}
The \emph{Temperley-Lieb space} $\TL(\Sigma)$ of a labelled surface $\Sigma$
is the set of all $\mathbb{Z}[q,q^{-1}]$ linear combinations of isotopy
classes of 1-manifolds (``multi-curves'') $F \subset \Sigma$ intersecting $\partial \Sigma$ transversely at marked points and subject to the local relation: removing a simple closed curve bounding a disk in $\Sigma$ from a multi-curve is equivalent to multiplying the resulting element by
the quantum integer $[2] = q + q^{-1}$.
\end{definition}

$\TL(\Sigma)$ is a 2-dimensional version of the Kauffman skein module of
${\Sigma}\times I$.  The \emph{Temperley-Lieb algebra} $\TL_n$ is given by
$\TL(D^2, \phi)$ where $\phi$ identifies $\partial D^2$ with a circle
containing $2n$ marked points. If we view $D^2$ as a rectangle with $n$
marked points on top and $n$ marked points on the bottom then vertical stacking
corresponds to the usual algebra structure on $\TL_n$.  The
construction of $\TL(\Sigma)$ is compatible with gluing surfaces along
common labels. In particular, the restriction to genus zero surfaces
 gives the structure of a planar algebra.

\subsection{The Jones-Wenzl Projectors} \label{jwproj}

For each $n \geq 1$ the Temperley-Lieb algebra $\TL_n$ contains special
idempotent elements called the \emph{Jones-Wenzl projectors}.  If we depict
elements of $\TL_n$ graphically with $n$ incoming and $n$ outgoing chords
then
\begin{equation} \label{projeq} \CPPic{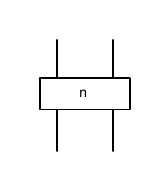} = \CPPic{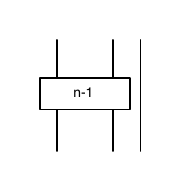} - \frac{[n-1]}{[n]}\! \CPPic{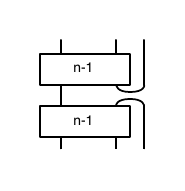} \end{equation}

where the first element $p_1 \in \TL_1$ is given by a single
chord. Jones-Wenzl projectors will be central to the discussion in
section \ref{arelcyc} and sections \ref{cosetcon}, \ref{handleslide}. For more detail, see
\cite{CK, KL}.

\subsection{Restricted Coefficients}\label{restr coeff}

Recall the quantum integers,
$$[n] = \frac{q^n-q^{-n}}{q-q^{-1}} = q^{n-1} + q^{n-3} + \cdots + q^{-n+3} + q^{-n+1}.$$

\begin{definition}\label{restricted ring}

Consider the power series
$$s_k = \sum_{i=0}^{\infty} (q^{(2i+1)k-1} - q^{(2i+1)k+1})$$

determined by the expansion of $1/[k]$ into positive powers of $q$. For each
$N >0$ the \emph{restricted ring}
$$R^N = \mathbb{Z}[q,q^{-1}][s_1,\ldots, s_{N}] \subset \mathbb{Z}[q^{-1}]\llbracket q \rrbracket$$

is obtained by adjoining $s_1,\ldots,s_N$ to the ring of Laurent
polynomials $\mathbb{Z}[q,q^{-1}]$.

\end{definition}

We have $R^1 = \mathbb{Z}[q,q^{-1}]$ and inclusions
$$R^1 \subset \cdots \subset R^N \subset R^{N+1} \subset \cdots \subset \mathbb{Z}[q^{-1}]\llbracket q \rrbracket.$$

\begin{definition}
Let $\Sigma$ be a labeled surface (see definition \ref{TL def}).  Recall that
$\TL(\Sigma)$ is the $\mathbb{Z}[q,q^{-1}]$ module given by the
Temperley-Lieb space. Define the \emph{$N$th restricted
  Temperley-Lieb space} by extension of coefficients,

$$\TL^N(\Sigma) = \TL(\Sigma) \otimes_{\mathbb{Z}[q,q^{-1}]} R^N.$$

\end{definition}

The restricted space $\TL^N$ is precisely the extension of $\TL$ necessary
to define the first $N$ Jones-Wenzl projectors: $p_1,p_2,\ldots,p_{N} \in
\TL^N$ (after $1/[k]$ is replaced by $s_k(q)$ in equation (\ref{projeq})).  Each
restricted space $\TL^N$ will have a categorified analogue $\Kom^N$ (see
section \ref{cateofchcx}) and these special coefficients will allow us to
control the quotient later.

We would like to consider the quotient of $\TL^N$ by the ideal $\inp{p_N}$.
Note that $1/\Tr(p_N) = 1/[N+1] \not\in R^N$, this motivates the definition
of the restricted ring $R^N$ above.

\begin{definition}\label{idealTL}
Let $Q\in \TL^N(D^2,\hat x)$. The \emph{ideal} $\inp{Q}$ generated by
$Q$ in $\TL^N(\Sigma)$ is the smallest submodule containing all elements
obtained by gluing $Q$ to $B$ where ${\Sigma}=(D^2,\hat
x)\cup({\Sigma}\smallsetminus D^2, \hat x)$ with $D^2\subset int({\Sigma})$,
and $B$ is any element of $\TL^N({\Sigma}\smallsetminus D^2,\hat x)$.
\end{definition}

\begin{figure}[h]
\begin{center}
\includegraphics[scale=.7]{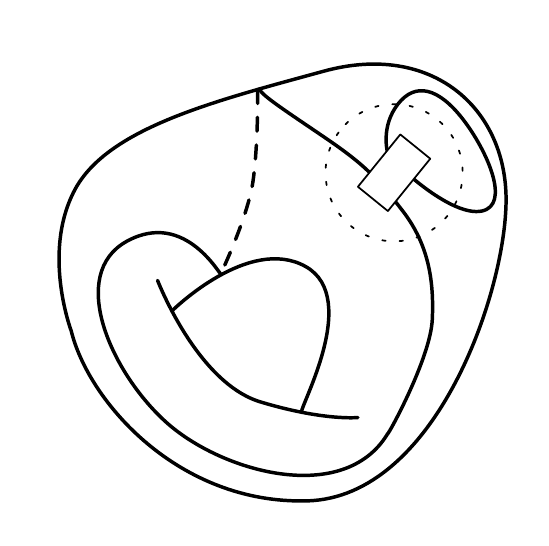}
\end{center}
\caption{An element of the ideal $\inp{Q}\subset
\TL^N(\Sigma)$ where $Q$ is represented by a box, and the surface $\Sigma$ is
a torus.}
\end{figure}
The notion of an ideal $\inp{Q}$ allows one to consider an element $Q$ of
the Temperley-Lieb algebra as a local relation ``$Q=0$'' among multi-curves
on a surface $\Sigma$, see \cite{Walker, FNWW} for more details.  The term
``ideal'' is motivated by the notion of an ideal in the planar algebra
setting (which corresponds to restricting our discussion to surfaces of
genus zero).

\subsection{Construction of $\SU(2)$ quantum invariants} \label{su2 invariants}

We start by briefly recalling a construction of the Turaev-Viro theory
\cite{TV} associated to a surface $\Sigma$, following \cite{FNWW}.  We will
then present a slight modification of the definition where the evaluation at
a root of unity is replaced by taking a suitable quotient by a Jones-Wenzl
projector.

Let $N=p-1$ and consider the Temperley-Lieb space $\overline{\TL}^N({\Sigma})$ with
coefficients in the ring $K^p =
\mathbb{Z}[q,q^{-1}]/\inp{\varphi_{p}(q^2)}$, where $\varphi_p(q)$ denotes
the $p$th cyclotomic polynomial. Then the Turaev-Viro $\SU(2)$ theory
associated to $\Sigma$, $\TV^N(\Sigma)$, is the quotient $\TL(\Sigma)$ by the ideal $\inp{p_{N}}$, where
$p_{N}$ denotes the $N$th Jones-Wenzl projector in the Temperley-Lieb
algebra. Note that our definition is somewhat different from that in \cite{FNWW},
specifically rather than setting $q$ to be a root of unity in the ground ring (which may be algebraically implemented
by taking a quotient by the cyclytomic polynomial $\varphi_p(q)$) we define the coefficient ring $K^p$ as the quotient by $\varphi_p(q^2)$.
The relation between the two is $\varphi_p(q^2) = \varphi_p(q)\varphi_p(-q)$.
The reason for our choice of the coefficient ring is explained in section \ref{arelcyc}.

The spaces $\TV^N$ suffice for the study of mapping class group
representations. Such spaces are determined in a canonical way by a
3-dimensional TQFT \cite{BHMV}. On the other hand in order to define a
3-dimensional TQFT using such spaces one must choose maps associated to the
addition of 3-dimensional handles which transform in accordance with handle slides and cancellations.
When such a choice exists it is unique up to multiplication by a
scalar \cite{C, KS}. In our context this information is determined by the
special element $\omega_N$ (see section \ref{rkh}).

For reasons explained in section \ref{categorified TL section}, it will be
convenient for us to consider {\em cosets} of the submodule $\inp{P_N}$
generated by $P_{N}$ in $\TL({\Sigma})$. Of course, the quotient
$\TL(\Sigma)/\inp{P_{N}}$ may be equivalently considered as the set of
cosets of $\inp{P_{N}}$ in $\TL(\Sigma)$. We stress this minor detail so it
is easier to follow the categorical construction in section
\ref{categorified TL section}.

\subsection{Evaluating at a root of unity by killing a projector}\label{arelcyc}

Set $N=p-1$ and recall the sequence of steps used in the definition of
$\TV^N({\Sigma})$ above: first the coefficients of $\TL({\Sigma})$ are set
to be $K^p$ (this is similar to setting the parameter $q$ to be a root of
unity), and then one takes the quotient of $\TL^N(\Sigma)$ by the submodule
$\inp{p_N}$. We will next show that in fact the first step may be omitted.

Consider the ring $K^p = \mathbb{Z}[q,q^{-1}]/\inp{\varphi_{p}(q^2)}$ as
above. Recall the restricted ring $R^N$ (definition \ref{restricted ring})
and consider the ideal $I =\inp{p_N}\subset \TL^N(S^2)=R^N$.

\begin{lemma}
Given a prime $p$, let $N=p-1$. Then

$$(R^N/I)\otimes \mathbb{Q} \cong K^p \otimes \mathbb{Q}.$$
\end{lemma}

\begin{proof}
Since $[k]$ is a unit in $K^p\otimes \mathbb{Q}$ the quotient map
$\mathbb{Z}[q,q^{-1}]\otimes\mathbb{Q} \to K^p \otimes \mathbb{Q}$ extends
to a map $R^N \otimes \mathbb{Q} \to K^p\otimes \mathbb{Q}$. Moreover, since
any element of $I$ evaluates to zero in $K^p \otimes
\mathbb{Q}$, this map descends to a map
$$\phi : (R^N/I) \otimes \mathbb{Q} \to K^p \otimes \mathbb{Q}.$$

On the other hand, observe that the cyclotomic polynomial $\varphi_p(q^2)$ is in the ideal $I$, since $\varphi_p(q^2) = q^N \tr(P_N)$. This gives an inverse map
$$\psi : K^p \otimes \mathbb{Q} \to (R^N/I)\otimes \mathbb{Q}.$$
\end{proof}

Setting the variable $q$ to a root of unity corresponds to quotienting the
ring of coefficients by the ideal $\inp{\varphi_p(q)}$. Since
$\inp{\varphi_p(q^2)}\subset\inp{\varphi_p(q)}$ the lemma implies

$$ \TL \longrightarrow  \TL/\inp{p_N}  \longrightarrow \TL/\inp{\varphi_p(q)}.$$

\begin{remark}
The proof of this lemma generalizes to give a ``picture TQFT'' construction
of $\SU(2)_N$ Turaev-Viro theory for a surface $\Sigma$ by killing $P_N$ in
the $\SU(2)$ Kauffman skein module of $\Sigma\times I$ over $R^N$.
\end{remark}

\subsection{Kirby Calculus and Handle Slides}\label{kirby}

It is a classical theorem of Lickorish and Wallace that any 3-manifold $M$ can be
obtained by surgery on a framed link $L \subset S^3$. Kirby calculus is a
means of describing when surgery on two different framed links produce
homeomorphic 3-manifolds:

\begin{theorem}{\cite{Kirby}}
Two 3-manifolds $M$ and $N$ obtained by surgery on $S^3$ along framed links
$L$ and $L'$ are diffeomorphic if and only if the associated link
diagrams $D(L)$ and $D(L')$ are related by a sequence of Kirby moves:
\begin{figure}[h]
$\BPic{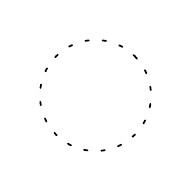} \cong \BPic{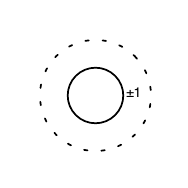} \normaltext{\quad \quad and \quad} \BPic{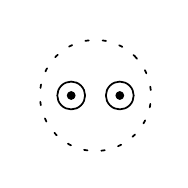} \cong \BPic{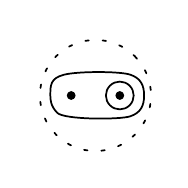}.$
\caption{}
\end{figure}
\end{theorem}

The first move describes the addition of a disjoint $\pm 1$-framed unknot to
the diagram.
It is the second, \emph{handle slide}, move that will be the focus of
this paper. The invariance under the first Kirby move is achieved by a suitable
normalization of the quantum invariant (cf. \cite{BHMV, Lickorish}). This is not
addressed in the present paper.

\subsection{The magic element} \label{rkh}

In \cite{RT} (see also \cite{KL}) a special
element $\omega_N$ is defined which can be represented by a linear combination of
projectors:

\begin{equation} \label{omega eq} \omega_N = \sum_{k=0}^{N}\,  \Tr(p_k) \phi_k. \end{equation}

Here $\phi_k$ is the element of the annulus category obtained by placing the
projector $p_k$ in the annulus. This $\omega_N$, sometimes referred to as
the magic linear combination of Temperley-Lieb elements, is the element
which is invariant under handle slides used in most constructions of quantum
3-manifold invariants.

\begin{lemma}
The element $\omega_N$ is invariant under handle slides in the module
$\TL^N(S^1\times I)/\inp{p_N}$.
\end{lemma}

A beautiful proof of this lemma was given by Lickorish (see
\cite{Lickorish}) when $q$ equals a root of unity. We observe that his proof remains true
in the annulus module above, where the evaluation at a root of unity is replaced by taking a
quotient by the ideal generated by the Jones-Wenzl projector.

The Reshetikhin-Turaev invariant of a $3$-manifold $M$ is defined by
presenting $M$ as a surgery on a framed link $L$ in $S^3$, labeling each
link component by the element $\omega_N$ and evaluating the Jones polynomial
of the resulting labeled link at the corresponding root of unity,
cf. \cite{KL, Lickorish}. To relate this to the context of the
Temperley-Lieb spaces and the Turaev-Viro theory $\TV^N(\Sigma)$ associated
to surfaces, note that the Jones polynomial of a link may be viewed as the
evaluation of the link in the skein module associated to $S^2$.  The
Turaev-Viro theory discussed above is the quantum double of the
Reshetikhin-Turaev TQFT, however in the case of the $2$-sphere both theories
are $1$-dimensional. Therefore it makes sense to analyze the handle-slide
property in the module $\TL^N(S^1\times I)/\inp{p_N}$ which was defined
above in the Turaev-Viro setting.

The reader should note that the omegas introduced in section
\ref{handleslide} are not precisely the same as the special element
$\omega_N$ defined above. The two definitions are compared in section
\ref{relationtomagic}.

\subsection{Spin $3$-manifolds} \label{spin}
A refined version of $\SU(2)$ quantum invariants may be defined for a $3$-manifold with a spin structure \cite{Blanchet}.
One considers a decomposition ${\omega}={\omega}_0+{\omega}_1$, corresponding to the ${\mathbb Z}/2$ grading (parity)
of the summation indexing set in (\ref{omega eq}). Given a presentation of a $3$-manifold $M$ with a spin structure as
the surgery on a framed link $L\subset S^3$, one labels the components of the characteristic sublink by ${\omega}_1$ and the rest
of the components by ${\omega}_0$. A spin analogue of the Kirby calculus is then used \cite{Blanchet} to prove that an appropriate
normalization is an invariant of $M^3$ with a given spin structure.

\section{Categorification of Temperley-Lieb spaces} \label{categorified TL section}

In this section we recall a version of Dror Bar-Natan's graphical
formulation \cite{DBN} of the Khovanov categorification \cite{KH}. It will
be used throughout the remainder of this paper. The Bar-Natan formulation
extends to planar algebras and surfaces in a transparent way and is
essentially universal (see \cite{KhF}) and so has the advantage of allowing
our constructions to apply to a number of variant categorifications which
exist in the literature.

Let $\Sigma$ be a labelled surface. There is an additive category
$\PCob(\Sigma)$ whose objects are isotopy classes of formally $q$-graded
1-manifolds $F \subset \Sigma$ intersecting all boundary components
transversely at marked points (compare to definition \ref{TL def}).

The morphisms are given by $\mathbb{Z}[\alpha]$ linear combinations of
isotopy classes of orientable cobordisms bounded in $\Sigma\times [0,1]$
between two surfaces $\Sigma$ containing such diagrams. Each cobordism is
constant along the boundary. The \emph{degree} of a cobordism $C : q^i A \to
q^j B$ is given by $\deg(C) = \deg_t(C) + \deg_q(C)$ where the topological
degree $\deg_t(C) = \chi(C) - n$ is given by the Euler characteristic of $C$
and the $q$-degree $\deg_q(C) = j - i$ is given by the relative difference
in $q$-gradings. Homological degree is not part of the definition $\deg(C)$.

It has become a common notational shorthand to represent a handle by a dot
and a saddle by a flattened diagram containing a dark line.
\begin{figure}[h]
$\CPPic{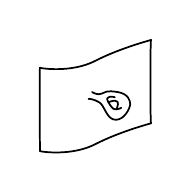}\! = 2\! \CPPic{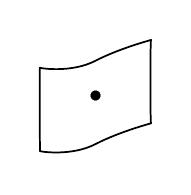} = 2 \CPPic{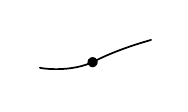} \textnormal{ and } \CPPic{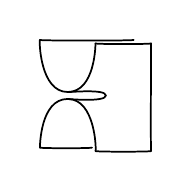}\,\, =\!\! \CPPic{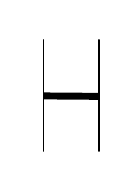}$
\caption{}
\end{figure}

(The topological degrees of the cobordisms above are $-2$, $-1$
respectively.)  When $(\Sigma,\phi) = (D^2,2n)$ we would like a category
$\mathcal{C}$ such that $\Ktheory_0(\mathcal{C}) \cong \TL_n$ so we require
that the object represented by a closed circle be isomorphic to sum of two
empty objects in degrees $\pm 1$ respectively. If such maps are to be degree
preserving then the most natural choice for these maps is given below.
\begin{figure}[h]
$\begin{diagram}
\varphi\!: \hspace{-.15in}
&\CPic{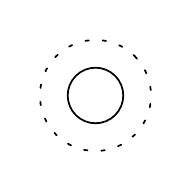}\hspace{.1in} & \pile{\rTo^{\left( \CPic{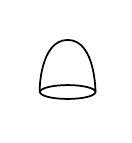} \CPic{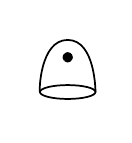} \right)  }\\
\lTo_{\left( \CPic{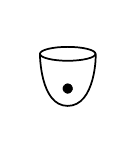} \CPic{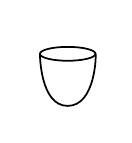} \right)  }}
& \hspace{.1in} q^{-1}\CPic{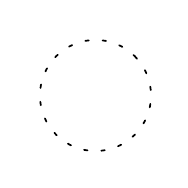}\, \oplus\, q \CPic{localblank} &\hspace{.15in} : \! \psi
\end{diagram}
$
\caption{}
\end{figure}

In order to obtain $\varphi \circ \psi = 1$ and $\psi \circ \varphi = 1$ we
form a new category $\Cob(\Sigma)$ obtained as a quotient of the
category $\PCob(\Sigma)$ by the local relations given below.

\begin{figure}[h]
$\CPic{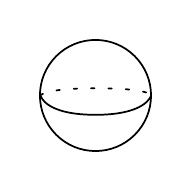} = 0 \hspace{.75in} \CPic{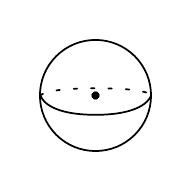} = 1 \hspace{.75in} \CPic{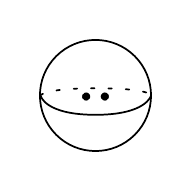} = 0 \hspace{.75in} \CPic{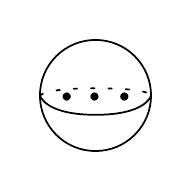} = \alpha$
$\CPic{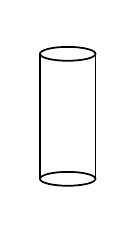} = \CPic{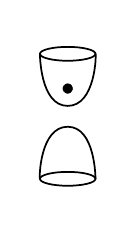} + \CPic{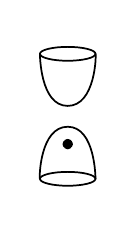}$
\caption{}
\end{figure}

The neck cutting relation can be applied to any compressing disk. In
particular, that this means the isomorphisms $\phi$ and $\psi$ cannot be
applied to remove a non-trivial circle from the annulus. The cylinder or
neck cutting relation implies that closed surfaces $\Sigma_g$ of genus $g >
3$ evaluates to $0$. The relations above imply that a sheet with two dots is
equal to $\alpha$ times a sheet with no dots. In what follows we will let
$\alpha$ be a free variable and absorb it into our base ring ($\Sigma_3=
8\alpha$). One can think of $\alpha$ as a deformation parameter, see
\cite{DBN} for further details.

\begin{proposition}
For any labelled surface $\Sigma$ there is an isomorphism of
$\mathbb{Z}[q,q^{-1}]$-modules,
$$\Ktheory_0(\Cob(\Sigma)) \cong \TL(\Sigma).$$
\end{proposition}

The proof follows directly from the construction above. Note that when
$\Sigma = D^2$ this is precisely the same setting as \cite{CK} and
\cite{CHK}. For an extended discussion of Grothendieck groups in this
context see \cite{CK}.

\begin{remark}
Surfaces with common boundary labels can be glued and the construction of
$\Cob(\Sigma)$ is compatible with this gluing.
\end{remark}

\subsection{Filtered Categories and Universal Projectors}\label{cateofchcx}

\begin{definition}\label{kom def}
If $\Sigma$ is a labelled surface let $\Kom^{\infty}(\Sigma) =
\Kom(\Mat(\Cob(\Sigma)))$ be the category of chain complexes of finite
direct sums of objects in $\Cob(\Sigma)$. In $\Kom^{\infty}$ we allow chain
complexes $K_*$ of unbounded positive homological degree and bounded
negative homological degree. Let $\Kom^{\infty}(n) = \Kom^{\infty}(D^2,2n)$
denote the category associated to the disk. Let $\Kom^1(\Sigma)$ be the
category of bounded complexes.
\end{definition}

In previous work the authors showed that the category $\Kom^{\infty}(n)$
contains special objects $P_n$ which categorify the Jones-Wenzl projectors
$p_n\in \TL_n$.

\begin{theorem}{\cite{CK}}
There exists a chain complex $P_n \in \Kom^{\infty}(n)$ called the
\emph{universal projector} which satisfies

\begin{enumerate}
\item $P_n$ is positively graded with degree zero differential.
\item The identity diagram appears only in homological degree zero and only
  once.
\item The chain complex $P_*$ is contractible under turnbacks.
\end{enumerate}

These three properties guarantee that $P_n$ is unique up to homotopy, $P_n
\otimes P_n \simeq P_n$ and $\Ktheory_0(P_n) = p_n \in \TL_n$.
\end{theorem}

Other authors \cite{FSS, Rozansky} constructed projectors using different techniques.

\begin{definition}
Given $N \in \mathbb{Z}_+$ and a labeled surface $\Sigma$, let
$\Kom^N(\Sigma)$ be the full subcategory of $\Kom^{\infty}(\Sigma)$
consisting of those chain complexes whose image under $\K_0$ is contained in
$\TL^N({\Sigma})$.
\end{definition}

By construction,
$$\Kom^1(\Sigma) \subset \cdots \subset \Kom^N(\Sigma) \subset \Kom^{N+1}(\Sigma) \subset \cdots \subset \Kom^{\infty}(\Sigma)$$

where $\Kom^1(\Sigma)$ is the category of bounded complexes and
$\Kom^{\infty}(\Sigma)$ is as in definition \ref{kom def}.

\section{Localization}\label{cosetcon}

The idea of localization is central to this paper.  We seek to carry out a
categorified analogue of the ``picture TQFT'' construction in sections
\ref{su2 invariants}, \ref{arelcyc}, and a central step in this program is
to kill $P_N$ in $\Kom^N(\Sigma)$ for all $\Sigma$. In order to accomplish
this we find a subcategory $\inpp{P_N}$ of $\Kom^N(\Sigma)$ the image of
which corresponds to the ideal $\inp{P_N}$ in the Grothendieck group
$\K_0(\Kom^N(\Sigma))\cong \TL^N(\Sigma)$.  A natural family of morphisms is
then inverted (definition \ref{localset}) in which we identify any object
$X$ in $\Kom^N(\Sigma)$ with any cone of $X$ on any object $P'\in
\inpp{P_N}$. This localization can be understood concretely (definition
\ref{coset definition}) in terms of iterated cones over the ideal
$\inpp{P_N}$. In section \ref{verdier} we explain why the usual notion of
Verdier localization is not suitable in our context, leading us to a
different version of localization described next.

In what follows, we fix $N > 0$ and use $P$ to denote $P_N$ and $\Kom$ to
denote $\Kom^N(\Sigma)$ for some $\Sigma$. Also $\Ho(\Kom)$ is the homotopy
category of $\Kom$.

\begin{definition}\label{ideal}
Given $Q$ in $\Ho(\Kom(D^2,\hat{x}))$, the \emph{ideal} $\inpp{Q} \subset \Ho(\Kom(\Sigma))$ generated by $Q$
is the smallest full subcategory of
$\Ho(\Kom({\Sigma}))$ which contains all objects obtained by gluing $Q$ to
$B$ where ${\Sigma}=(D^2,\hat x)\cup({\Sigma}\smallsetminus D^2, \hat x)$
with $D^2\subset int({\Sigma})$, $B$ is any object of
$\Ho(\Kom({\Sigma}\smallsetminus D^2,\hat x))$, and which is closed under
cones and grading shifts.
\end{definition}

The reader should compare this to definition \ref{idealTL}.
\begin{figure}[h]
$\CPPic{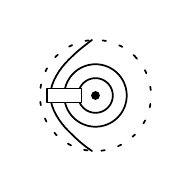}, \quad \CPPic{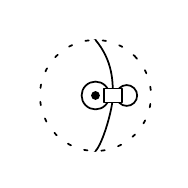}.$
\caption{Examples of objects in the ideal $\inpp{P_3}$ in $\Kom(S^1\times
I)$, where the annulus $S^1\times I$ has two marked points on the boundary.}
\end{figure}

Notice that $A\oplus B \in \inpp{P} \not\Rightarrow A\in \inpp{P}$; the
subcategory $\inpp{P}$ is not thick, compare with section \ref{verdier}. The
definition is chosen so that the following proposition holds.

\begin{proposition}
If $\Ho(\Kom)$ and $P \in \Ho(\Kom)$ are as above, then
$$\Ktheory_0(\inpp{P}) = \inp{P}$$

where $\inp{P} \subset \Ktheory_0(\Ho(\Kom))\cong \TL^N$ is the ideal generated by $P$ (definition \ref{idealTL}).
\end{proposition}

The proof follows from definition \ref{ideal}. 
For more on $\K_0$ of infinite complexes see \cite{CK}. For a discussion of
$\K_0$ of triangulated categories see \cite{W2}.

\begin{proposition}{(Localization of a category)}\label{loccat}
Given a collection $S$ of morphisms of a category $\cal{C}$,
there exists a category $\cal{C}[S^{-1}]$ and a functor $Q : \cal{C}\to
\cal{C}[S^{-1}]$ satisfying the universal property:
\begin{itemize}
\item If $F : \cal{C} \to \cal{D}$ is a functor and $F(s)$ is an isomorphism
  for all $s\in S$ then there exists a functor $G : \cal{C}[S^{-1}] \to
  \cal{D}$ such that $F = G \circ Q$.
\end{itemize}
\end{proposition}

The category $\cal{C}[S^{-1}]$ is constructed by realizing $\cal{C}$ as a
directed graph, adding edges corresponding to inverses of elements of $S$
and taking morphisms to be paths in this graph subject to the most obvious
relations. For details see \cite{GM}.

\begin{definition}{(``brusque quotient'')}\label{localset}
Given $P'\in \inpp{P} \subset \Ho(\Kom)$, if $f : X\to P'$ or $g : P' \to X$
there are maps $\pi_f : \Cone(f) \to X$ and $i_g : X \to \Cone(g)[-1]$. We
wish to invert all such maps. Set
$$S = \{ \pi_f| f : X\to P',\, P' \in \inpp{P} \} \cup \{ i_g| g : P'\to X,\, P' \in \inpp{P} \},$$
and define
$$\Ho(\Kom)/\inpp{P} = \Ho(\Kom)[S^{-1}].$$
\end{definition}

Note that this version of localization has some unusual properties. For instance,
$\Id_{A\oplus B} \cong 0 \not\Rightarrow \Id_{A} \cong 0$ in the quotient, and the familiar
triangulated structure is now gone.
Section \ref{verdier} explains how this definition is different from the usual Verdier localization.

We will now show that the above construction can be reformulated in a more
concrete way. (The reader may note that the theme of characterizing localization in terms of iterated
extensions appears in algebraic topology literature, cf. \cite{Chacholski}). We begin with an
equivalence relation on objects in $\Kom$.

\begin{definition}{($\inpp{P}-$equivalence)}
Let $\inpp{P}\subset\Ho(\Kom)$ be as above then for any $A,B\in\Ho(\Kom)$ we
say that $A$ is $\inpp{P}$-\emph{equivalent} to $B$ if there exists a sequence of
maps $\{f_i\}_{i=1}^m$ and objects $Q_i \in \inpp{P}$ such that

\begin{enumerate}
\item $f_1 : B \to Q_1$ or $f_1 : Q_1 \to B$ for some $Q_1 \in \inpp{P}$
\item $f_i : \Cone(f_{i-1}) \to Q_i$ or $f_i : Q_i \to \Cone(f_{i-1})$
\end{enumerate}

and $\Cone(f_m)\cong A$
\end{definition}

This equivalence relation amounts to isomorphism in the category
$\Ho(\Kom)/\inpp{P}$ defined above:

\begin{proposition}
$A\cong B$ in $\Ho(\Kom)/\inpp{P}$ if and only if $A$ is
  $\inpp{P}$-equivalent to $B$ in $\Ho(\Kom)$.
\end{proposition}

\begin{proof}
The first direction is obvious. Suppose $A\cong B$ in $\Ho(\Kom)/\inpp{P}$
then by definition an isomorphism $\phi : A \to B$ is a map of the form
$\phi = h_n \phi_n h_{n-1} \phi_{n-1} \cdots h_1 \phi_1$ where $\phi_k$ is
an isomorphism in $\Ho(\Kom)$ and $h_k$ is an isomorphism of the form
$\pi_{f_k}^{\pm 1}$ or $i_{g_k}^{\pm 1}$ as in definition \ref{localset}
above. A map $h_k : X_k \to X_{k+1}$ tells us that either $X_k$ is
equivalent to $X_{k+1}$ coned on $P'\in \inpp{P}$ or vice versa.
\end{proof}

Given an object $A \in \Ho(\Kom)$ we can combine all of the objects which
are $\inpp{P}$-equivalent to it to form a category which is a natural analogue of
the coset $A + \inp{P}$ in the image of the Grothendieck group.

\begin{definition}{(Coset categories)} \label{coset definition}
Given $A\in\Ho(\Kom)$, the \emph{coset} associated to $A$ is the
full subcategory, denoted $A + \inpp{P} \subset \Ho(\Kom)$, consisting of objects $Y$ which
are $\inpp{P}$-equivalent to $A$ in $\Ho(\Kom)$.
\end{definition}

Isomorphism classes of coset subcategories are in bijection with
$\inpp{P}$-equivalence classes:

\begin{proposition}
$A$ is $\inpp{P}$-equivalent to $B$ if and only if $A + \inpp{P} \cong B +
  \inpp{P}$
\end{proposition}

Notice that the Grothendieck
group of a coset category $\K_0(A + \inpp{P})$ can be identified with the
coset $\K_0(A) + \inp{P}$. In particular, the non-triviality of
$\TL^N/\inp{P_N}$ implies that the quotient $\Ho(\Kom)/\inpp{P}$ is highly
non-trivial.
More precisely, in light of remark at the end of section \ref{arelcyc}, the machinery developed
in this section may be viewed as a categorification of the $2$-dimensional part of the Turaev-Viro theory.

\begin{remark}
In \cite{KhH} Khovanov considered additive monoidal categories
$(\cal{C},\otimes,1)$ of graded objects in which $1 \oplus q1 \oplus \cdots
\oplus q^{n-1}1 \cong 0.$ Such an isomorphism implies that the polynomial
$1+q + \cdots + q^{n-1}$ is equal to $0$ in the Grothendieck group
$\K_0(\cal{C})$.

In this paper, $\Tr(P_N)$ is isomorphic to a chain complex consisting only
of empty diagrams \cite{CK} and the relation $\Tr(P_N) = 0$ implies that the
polynomial $1+q^2 + \cdots + q^{2N}$ is zero in the Grothendieck group.

There are two other more notable differences between these ideas. In our
context, the quotient must be homotopy invariant because $P_N$ is only
determined up to homotopy. Also inasmuch as $P_N$ is not determined by its
graded Euler characteristic, the effect of the quotient on the category has
less to do with the polynomial $\K_0(\Tr(P_N))$ and more to do with the
structure of $P_N$ itself.

\end{remark}

\subsection{On Verdier Quotients}\label{verdier}

In section \ref{cosetcon} we quotiented $\Kom^N(\Sigma)$ by $P_N$ in an
unusual way. In this section we recall elements of the most common quotient
construction for triangulated categories. In particular, we explain why it
is not useful in the context of this paper. Readers interested in a detailed
discussion of the Verdier quotient should consult \cite{GM, Neeman}.

Given a triangulated category $\cal{C}$ we would like to invert some family
of maps $T$ so that the objects of a full triangulated subcategory $\cal{D}
\subset\cal{C}$ become isomorphic to zero in a triangulated quotient
category $\cal{C}[T^{-1}] = \cal{C}/\cal{D}$.

In order to produce a triangulated quotient category $\cal{C}[T^{-1}]$ the
maps in $T$ are required to admit a calculus of left and right
fractions. The collection $T$ is additionally required to be closed under
homological grading shifts and if $(f,g,h)$ is a map between triangles with
$f,g\in T$ then there exists an $h'\in T$ so that $(f,g,h')$ is a map
between triangles. Most of these conditions are not imposed by the brusque
quotient of section \ref{cosetcon}.

If $\cal{D}$ is a full triangulated subcategory of a triangulated category
$\cal{C}$, then $T$ is defined to be the collection of maps $f : X \to Y$
which fit into an exact triangle:

$$X\to Y\to Z \to X[1]$$

where $Z\in \cal{D}$. If the collection $T$ is chosen in this manner then it
satisfies all of the properties described above. The quotient functor $Q :
\cal{C} \to \cal{C}[T^{-1}]$ is a triangulated functor and universal in the
category of small triangulated categories. The quotient functor $F$ in
proposition \ref{loccat} does not satisfy these properties.

\begin{definition} A \emph{thick} (or {\em \'epaisse}) subcategory $\cal{D}\subset\cal{C}$ is a subcategory which is closed under retracts: $A \oplus B \in \cal{D} \Rightarrow A \in \cal{D}$.
\end{definition}

The Verdier quotient only sees the smallest thick subcategory containing the
subcategory under consideration. More specifically, if
$\cal{D}\subset\cal{C}$ is a subcategory and $\bar{\cal{D}}$ is the smallest
thick subcategory containing $\cal{D}$ then $\cal{C}/\cal{D} \cong \cal{C}/\bar{\cal{D}}.$

The following proposition explains our problem.

\begin{proposition}
The smallest thick subcategory containing the ideal $\inpp{P_2}$ is
$\Ho(\Kom(\Sigma))$.
\end{proposition}
\begin{proof}
In \cite{CK} it was shown that $\tr(P_2) \simeq q^{-2} \mathbb{Z} \oplus \mathbb{Z} \oplus \ldots$.
\end{proof}

It follows that $\Ho(\Kom(\Sigma))/\inpp{P_2} \cong 0$. The following
general statement also holds.

\begin{proposition}
If $L$ is any link or $(1,1)-$tangle then the smallest thick subcategory containing the ideal $\inpp{L}$ is $\Ho(\Kom(\Sigma)).$
\end{proposition}
\begin{proof}
The chain complex associated to $L$ is homotopy equivalent to a chain
complex in which the Rasmussen invariant can be computed from the grading of
a trivial direct summand (see \cite{SM}).
\end{proof}

A conceptual explanation could be formulated by instead considering the
category $\Der^{\perf}(H^*(\mathbb{C}P^1)\module)$ (compare \cite{KH}). Here
thick subcategories are in correspondence with closed subschemes \cite{PB}
and $\Spec(\mathbb{Z}[x]/(x^2))$ has no interesting subschemes.

Other quotients such as those of Bousfield and Drinfeld also respect
thickness in the manner described above. It seems that the intricacies of
these homotopy theories are not seen by the usual quotients.

\section{Handle Slides} \label{handleslide}

In this section we use the framework discussed above to define objects
$\Omega$ which are invariant under handle slides.  These objects categorify
the magic elements ${\omega}_N$ in the annular Temperley-Lieb space for low
levels $N$.  Recall that ${\omega}_N$ is a crucial ingredient in the
construction of the $\SU(2)$ TQFT at level $N$, see section \ref{rkh}. The
construction of these objects relies on a detailed analysis of the universal
projector complexes from \cite{CK}.

\subsection{Handle Slides and The Second Projector} \label{second projector section}

The second projector is defined \cite{CK} to be the chain complex
$$\begin{diagram}
\CPic{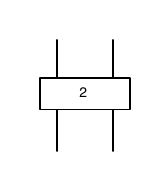}\! & = & \!\CPic{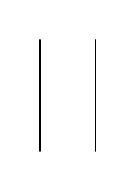} & \rTo^{\MPic{n2-1-sad}} &q \CPic{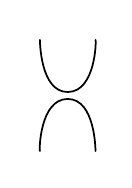} &\rTo^{\MPic{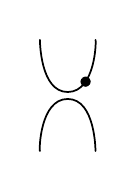} \!\!\!- \MPic{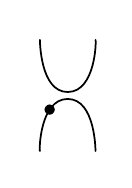}} &q^{3} \CPic{n2-s} &
\rTo^{\MPic{n2-tops}\!\!\! + \MPic{n2-bots}} & q^{5} \CPic{n2-s} &  \cdots
\end{diagram}$$

in which the last two maps alternate ad infinitum. More explicitly,

$$P_2 = (C_*, d_*)$$

The chain groups are given by
$$
C_n =
\left\{
\begin{array}{lr}
q^0 \MPic{n2-1}&n = 0\\
q^{2n-1} \MPic{n2-s}&n > 0
\end{array}
\right.
$$

The differential is given by

$$
d_n =
\left\{
\begin{array}{llr}
\MPic{n2-1-sad}                 & : \MPic{n2-1}\!\! \to q \MPic{n2-s}   & n = 0\\
\MPic{n2-tops} \!\!\! + \MPic{n2-bots} & :   q^{4k-1} \MPic{n2-s}\!\!  \to q^{4k +1} \MPic{n2-s}                    &n \ne 0, n = 2k\\
\MPic{n2-tops} \!\!\! - \MPic{n2-bots} & : q^{4k +1} \MPic{n2-s}\!\!  \to q^{4k+3} \MPic{n2-s}                    & n = 2k+1
\end{array}
\right.
$$

\begin{theorem}{\cite{CK}}\label{secondproj thm}
  The chain complex $P_2 \in \Kom(2)$ defined above is a universal projector.
\end{theorem}

We define the \emph{tail} of $P_2$ to be the chain complex

$$\begin{diagram}
\CPicTailbox{n2-s} & =  & q \CPic{n2-s} &\rTo^{\MPic{n2-tops} \!\!\!- \MPic{n2-bots}} & q^{3} \CPic{n2-s} & \rTo^{\MPic{n2-tops}\!\!\! + \MPic{n2-bots}} & q^{5} \CPic{n2-s} &  \cdots
\end{diagram}$$

so that

$$\CPic{p2box} = \Cone\left(\hspace{-.02in}\CPic{n2-1}\hspace{-.05in} \to \CPicTailbox{n2-s}\right) \normaltext{\hspace{.08in} and \hspace{.08in}} \CPic{n2-1}\hspace{-.09in} \simeq \Cone\left(\CPicTailbox{n2-s} \to \CPic{p2box}\right)$$

where the map defining the first cone is the saddle, $d_0$, appearing in the definition of
$P_2$ and the map in the second cone is the inclusion of the tail into $P_2$.

\begin{definition}{($\Omega_2$)} \label{omega2}
In the category $\Kom^2(S^1\times [0,1])$ let $\Omega_+$ be the
object consisting of a single curve about the origin and $\Omega_-$ be the
object consisting of a single nullhomotopic simple closed
curve. Let $\Omega = \Omega_+ \oplus \Omega_-$, we represent $\Omega$ by a curve
decorated with a dot. Graphically,
$$\Omega_+ = \CPPic{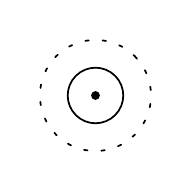}, \quad \Omega_{-} = \CPPic{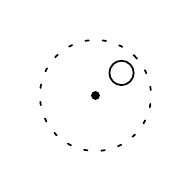}, \quad
\CPPic{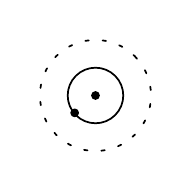} = \CPPic{omega2p} \oplus \CPPic{omega2m}.$$
\end{definition}

\begin{theorem} \label{omega2 thm}
The coset $\bar{\Omega} = \Omega + \inpp{P_2}$ associated to the object $\Omega$ defined above is invariant under handle slides.
\end{theorem}

\begin{proof}
Consider first the partial trace of $P_2$ in $\Kom^2(S^1\times [0,1])$,
$$\CPPic{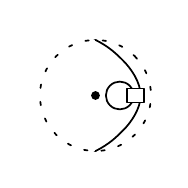} = \Cone\left(\hspace{-.05in}\CPPic{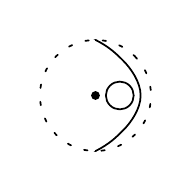}\hspace{-.05in} \to \CPPicTailbox{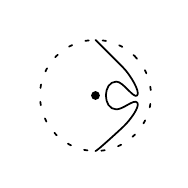}\right).$$

As above, the square brackets denote the tail of the second projector. Consider also the following diagram:
$$\CPPic{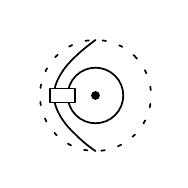} = \Cone\left(\hspace{-.05in}\CPPic{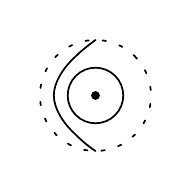}\hspace{-.05in} \to \CPPicTailbox{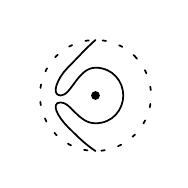}\right).$$

Note that the tail of the first equation and the tail of the second equation
above are equal chain complexes. Also the left-hand sides of both equations
are in the ideal $\inpp{P_2}$ (definition \ref{ideal}), since they contain the second projector. These two observations taken together
imply that the equation
$$\CPPic{omega2eq1b} \cong \CPPicTailbox{omega2eq1c} = \CPPicTailbox{omega2eq2c} \cong \CPPic{omega2eq2b}$$

holds among $\inpp{P_2}$ cosets. For example, the first isomorphism holds since
the diagram on the left is a cone of the inclusion of the tail into an element of $\inpp{P_2}$. This says that a single strand placed on
the lefthand side of $\Omega_+$ is isomorphic to a single strand placed on
the righthand side of $\Omega_-$ as $\inpp{P_2}$ cosets and vice versa.
This equation together with its vertical reflection
imply that
$$\CPPic{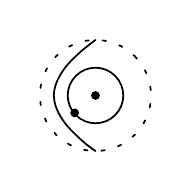} \cong \CPPic{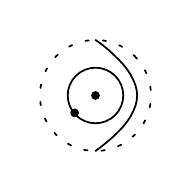}$$

as $\inpp{P_2}$ cosets, establishing the handle-slide property of $\Omega$.
\end{proof}

\subsection{Handle Slides and The Third Projector} \label{third projector section}

We start this section by recalling the minimal (in the sense that it is not homotopic to
a chain complex with fewer Temperley-Lieb diagrams) chain complex for
$P_3$, introduced in \cite{CK}. Then we show how it can be used to define an element $\Omega$ in the annular category
$\Kom^3(S^1\times [0,1])$ which is invariant under handle slides in the
associated coset category.

Recall that the third projector is given by

$$\begin{diagram}
\MPic{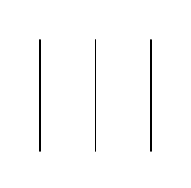}\, & \rTo^A &\, q^{1}\left(\! \MPic{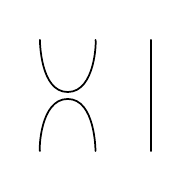}\,\, \oplus \MPic{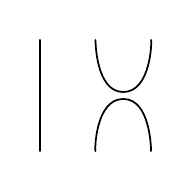}\!\right) & \rTo^B &\, q^{2}\left(\! \MPic{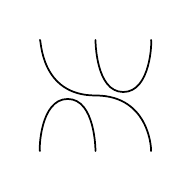}\,\, \oplus \MPic{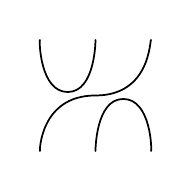}\!\right) & \rTo^C &\, q^{4}\left(\! \MPic{n3-e1e2}\,\, \oplus \MPic{n3-e2e1}\!\right) \\
  &  &           &   &         &   &  \dTo^D\\
\cdots & \lTo & \, q^{8}\left(\! \MPic{n3-e1e2}\,\, \oplus \MPic{n3-e2e1} \!\right) & \lTo^B & \, q^{7}\left(\! \MPic{n3-e1}\,\, \oplus \MPic{n3-e2}\!\right) & \lTo^E & \, q^{5}\left(\MPic{n3-e1}\,\, \oplus \MPic{n3-e2}\!\right) \\
\end{diagram}$$

Where $A = \left(\begin{array}{cc} \! -\MPic{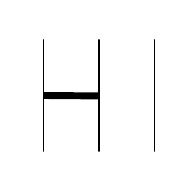} & \MPic{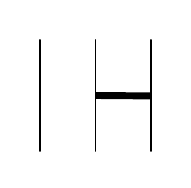} \end{array}\right)^{\rm tr}$ and

\begin{align*}
B &= \left(\begin{array}{cc}
\MPic{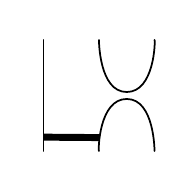} & - \MPic{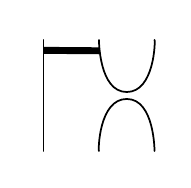}\\
-\MPic{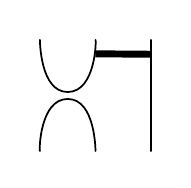} & \MPic{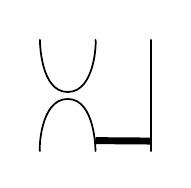}\\
\end{array}
\right)  &
C &= \left(\begin{array}{cc}
\MPic{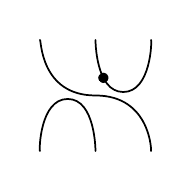} + \MPic{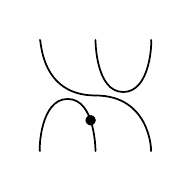} &  \MPic{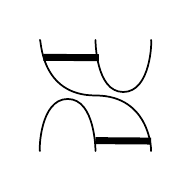}\\
\MPic{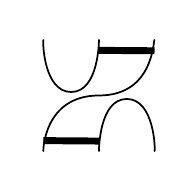} & \MPic{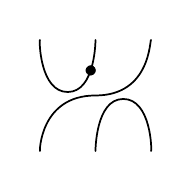} + \MPic{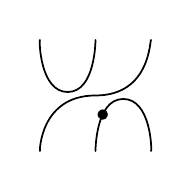}\\
\end{array}
\right) & \\
D &= \left(\begin{array}{cc}
\MPic{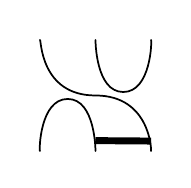} & - \MPic{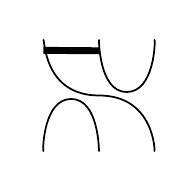}\\
-\MPic{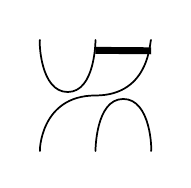} & \MPic{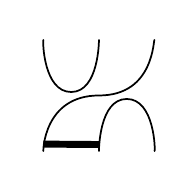}\\
\end{array}
\right) &
E &= \left(\begin{array}{cc}
\MPic{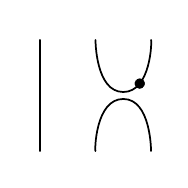} + \MPic{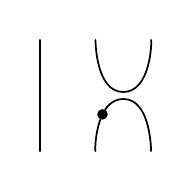} &  \MPic{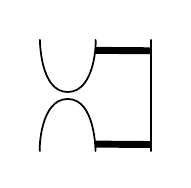}\\
\MPic{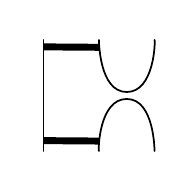} & \MPic{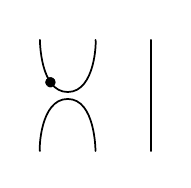} + \MPic{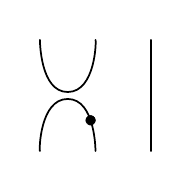}\\
\end{array}
\right) & \\
\end{align*}

After the initial identity term the complex becomes $4$ periodic.

\begin{theorem}{\cite{CK}}
  The definition of $P_3$ given above is a chain complex that satisfies the
  axioms of the universal projector.
\end{theorem}

We define the \emph{tail} of $P_3$ to be the chain complex

\newcommand{\ThreeTailbox}{\ensuremath{\Bigg[\!\MPic{n3-e1}\hspace{-.04in} \oplus\!\! \MPic{n3-e2}\! \rightarrow \! \MPic{n3-e1e2}\hspace{-.04in} \oplus\!\! \MPic{n3-e2e1}\!\Bigg]}}

$$\ThreeTailbox = \Cone\left(\MPic{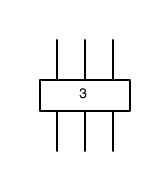} \to \MPic{n3-1}\right)$$

so that

$$\ThreeTailbox \simeq q^{1}\left(\! \MPic{n3-e1}\,\, \oplus \MPic{n3-e2}\!\right) \to  q^{2}\left(\! \MPic{n3-e1e2}\,\, \oplus \MPic{n3-e2e1}\!\right) \to \cdots$$

i.e. the ``tail'' of the chain complex for $P_3$ given at the beginning of this section.

\begin{definition}{($\Omega_3$)} \label{omega3}
In the category $\Kom^3(S^1\times [0,1])$ let $\Omega_+$ be the object
consisting of two curves about the origin and $\Omega_-$ be the object
consisting of a single curve about the origin and a single nullhomotopic
simple closed curve. Let $\Omega = \Omega_+ \oplus \Omega_-$, we represent $\Omega$ by a curve
decorated with a dot then in terms of pictures. Graphically,
$$\Omega_+ = \CPPic{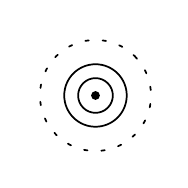},\quad \Omega_{-} = \CPPic{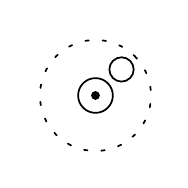}, \quad
\CPPic{omega} = \CPPic{omega3p} \oplus \CPPic{omega3m}.$$
\end{definition}

\begin{theorem} \label{omega3 thm}
The coset $\bar{\Omega} = \Omega + \inpp{P_3}$ associated to the object $\Omega$ defined above is invariant under handle slides.
\end{theorem}

\begin{proof}
The proof consists of analyzing the partial traces of $P_3$ in the annulus.
\newcommand{\ThreeTailboxeqone}{\ensuremath{\Bigg[\!\CPPic{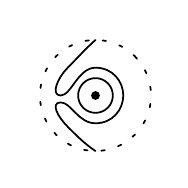}\hspace{-.04in} \oplus\!\! \CPPic{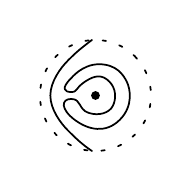}\! \rightarrow \! \CPPic{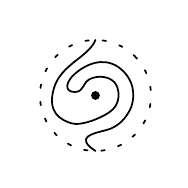}\hspace{-.04in} \oplus\!\! \CPPic{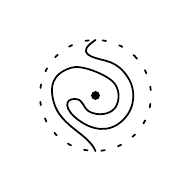}\!\Bigg]}}

$$\CPPic{omega3eq1a} = \CPPic{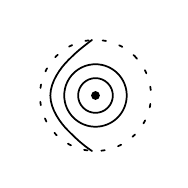}\hspace{-.03in}\to \ThreeTailboxeqone$$

where we omit the word $\Cone$ above to save horizontal space. Consider the partial trace,
\newcommand{\ThreeTailboxeqtwo}{\ensuremath{\Bigg[\!\CPPic{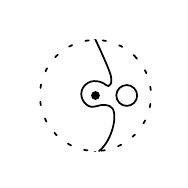}\hspace{-.04in} \oplus\!\! \CPPic{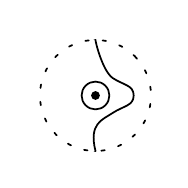}\! \rightarrow \! \CPPic{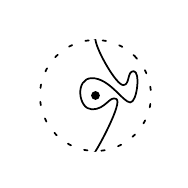}\hspace{-.04in} \oplus\!\! \CPPic{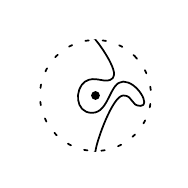}\!\Bigg]}}
$$\CPPic{omega3eq2a} = \CPPic{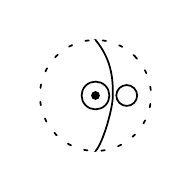}\hspace{-.03in}\to \ThreeTailboxeqtwo.$$

Notice that the two tails above are the same up to changing the order of the summands,

\newcommand{\MThreeTailboxeqone}{\ensuremath{\Bigg[\!\CPic{omega3eq1tailbox1}\hspace{.05in} \oplus\!\! \CPic{omega3eq1tailbox2}\! \rightarrow \! \CPic{omega3eq1tailbox3}\hspace{.05in} \oplus\!\! \CPic{omega3eq1tailbox4}\!\Bigg]}}
\newcommand{\MThreeTailboxeqtwo}{\ensuremath{\Bigg[\!\CPic{omega3eq2tailbox1}\hspace{.05in} \oplus\!\! \CPic{omega3eq2tailbox2}\! \rightarrow \! \CPic{omega3eq2tailbox3}\hspace{.05in} \oplus\!\! \CPic{omega3eq2tailbox4}\!\Bigg]}}
$$\MThreeTailboxeqone \cong \MThreeTailboxeqtwo$$

Since the left hand side of both equations above include the
third projector they are both elements of the ideal $\inpp{P_3}$. It follows
that
$$\CPPic{omega3eq1b} \cong \CPPic{omega3eq2b},$$

but this equation says that a free strand can be slid over $\Omega_+$ at
the cost of turning the $\Omega_+$ into an $\Omega_-$. This equation together with its vertical reflection
imply that the $\inpp{P_3}$ coset associated
to the $\Omega$ defined above is invariant under handle slides.
\end{proof}

\subsection{Relation to the magic element}\label{relationtomagic}
Here we compare the $\omega$ elements categorified in \ref{second projector
  section} and \ref{third projector section} to the special $\omega_N$ of
section \ref{rkh}. We use $X$ to represent a single essential circle in the
annulus and the relation $X \phi_k = \phi_{k+1} + \phi_{k-1}$ below.

In section \ref{second projector section} when $p_2 = 0$ we set $\omega = X
+ [2]$. It follows that $X \omega = X^2 + [2] X = \phi_2 + \phi_0 + [2]
\phi_1 = \omega_2$. So, $\omega_2 = X \omega$. The handle slide property
implies that $\omega_2 = [2]\omega$.

In section \ref{third projector section} when $p_3 = 0$ we set $\omega = X^2
+ [2] X$. If $\omega_3 = 1 + [2] \phi_1 + [3] \phi_2$ then $X\omega_3 =
([3]+1) X + [2] X^2$ since $X = X\phi_2$ when $p_3 = 0$ and $[3]+1=[2]^2
\Rightarrow X\omega_3 = [2] \omega$. Since $\omega_3$ satisfies the
handle slide property, $[2] \omega_3 = X \omega_3 = [2]\omega \Rightarrow \omega_3 =
\omega$ since $[2]$ is a unit in $\TL^3$ by construction.

\subsection{Spin manifolds} One may use
definitions \ref{omega2}, \ref{omega3} in the context of $3$-manifolds with
a spin structure. Specifically, one labels the components of the
characteristic sublink by ${\Omega}_-$ and the rest of the components by
${\Omega}_+$, compare with section \ref{spin}. Theorems \ref{omega2 thm},
\ref{omega3 thm} then state that sliding a strand interchanges the cosets
associated to ${\Omega}_-$ and ${\Omega}_+$, as one expects in the context
of spin Kirby calculus.

\subsection{Summary} This section introduced a categorification of the elements ${\omega}_2, {\omega}_3$, where ${\omega}_N$ is defined
in (\ref{omega eq}). It is shown that
these objects are invariant under handle slides, understood in the context
of the coset construction of section \ref{cosetcon}. The authors conjecture
that there is also a categorification of ${\omega}_N$ for all $N$, however
it is likely that a construction would require a further non-trivial
extension of ideas presented here.

One can construct an object in the brusque quotient category (section
\ref{cosetcon}) by labelling each component of a framed link $L\subset S^3$
with $\Omega$. Theorems \ref{omega2 thm} and \ref{omega3 thm} imply that the
isomorphism type of this object is unchanged by handle slides. This reflects
information about the 3-manifold obtained by surgery on $L$ (see section
\ref{kirby}). We expect that there are maps between such objects associated
to $4$-dimensional cobordisms.

{\bf Acknowledgements.} We would like to thank Nicholas Kuhn and Justin Roberts for helpful conversations on the subject of the paper.

B. Cooper was supported in part by the Max Planck Institute for Mathematics in Bonn.
V. Krushkal was supported in part by NSF grant DMS-1007342.

\end{document}